\documentclass[11pt, reqno]{amsart}
\usepackage{amsmath,amsthm,amssymb,mathrsfs}
\usepackage[abbrev,non-sorted-cites]{amsrefs}
\usepackage{color,graphicx}
\usepackage{verbatim}
\usepackage{datetime}
\usepackage{hyperref}

\theoremstyle{plain}
  \newtheorem{thm}{Theorem}[section]
  
  \newtheorem{prop}[thm]{Proposition}
   
	\newtheorem*{thm*a}{Theorem A}
	\newtheorem*{thm*b}{Theorem B}
\theoremstyle{definition}
  \newtheorem{defn}[thm]{Definition}
	
  \newtheorem{rmk}[thm]{Remark}
  \newtheorem*{ack*}{Acknowledgement}
  \newtheorem*{ques*}{Question}
\theoremstyle{plain}

\numberwithin{equation}{section}

\newcommand\ip[2]{\langle{#1},{#2}\rangle}

\newcommand\pl{\partial}

\newcommand\oh{\frac{1}{2}}
\newcommand\dd{{\mathrm d}}

\newcommand\dt{\delta}

\newcommand\om{\omega}
\newcommand\ta{\theta}

\newcommand\af{\alpha}

\newcommand\ld{\lambda}

\newcommand\Om{\Omega}

\newcommand\Ld{\Lambda}

\newcommand\Dt{\Delta}

\newcommand\CU{\mathcal{U}}

\newcommand\BC{\mathbb{C}}

\newcommand\BR{\mathbb{R}}
\newcommand\BN{\mathbb{N}}

\newcommand\td{\tilde}
\newcommand\br{\bar}

\newcommand\bx{\mathbf{x}}

\newcommand\bfI{\mathbf{I}}

\DeclareMathOperator{\vol}{vol}

\newcommand\ST{S^{[2]}}
\newcommand\heat{\frac{\pl}{\pl t} - \Dt}

\begin{document}

\title{Mean Curvature Flows of Two-Convex Lagrangians}

\author{Chung-Jun Tsai}
\address{Department of Mathematics, National Taiwan University, and National Center for Theoretical Sciences, Math Division, Taipei 10617, Taiwan}
\email{cjtsai@ntu.edu.tw}

\author{Mao-Pei Tsui}
\address{Department of Mathematics, National Taiwan University, and National Center for Theoretical Sciences, Math Division, Taipei 10617, Taiwan}
\email{maopei@math.ntu.edu.tw}

\author{Mu-Tao Wang}
\address{Department of Mathematics, Columbia University, New York, NY 10027, USA}
\email{mtwang@math.columbia.edu}

\date{\usdate{\today}}

\thanks{C.-J.~Tsai is supported by NSTC grant 110-2636-M-002-007, 111-2636-M-002-022, 112-2636-M-002-003 and NCTS.  M.-P.~Tsui is supported by NSTC grant 109-2115-M-002-006. This material is based upon work supported by the National Science Foundation under Grant
Numbers DMS-1810856 and DMS-2104212 (Mu-Tao Wang).  Part of this work was carried out when M.-T.~Wang was visiting the National Center of Theoretical Sciences.}

\begin{abstract} We prove regularity, global existence, and convergence of Lagrangian mean curvature flows in the two-convex case \eqref{2pos}.  Such results were previously only known in the convex case, of which the current work represents a significant improvement. The proof relies on a newly discovered monotone quantity \eqref{log_det} that controls two-convexity.  Through a unitary transformation, same result for the mean curvature flow of area-decreasing Lagrangian submanifolds \eqref{area-d} were established. 

\end{abstract}

\maketitle


\section{Introduction}

Let $M$ be a $2n$ dimensional K\"ahler manifold.  Throughout this paper, the Riemannian metric on $M$ is assumed to be flat.  The symplectic form $\omega$ on $M$ is given by $\om(\,\cdot\, , \,\cdot\,) = \ip{J(\,\cdot\,)}{\,\cdot\,}$ where $J$ is the (almost) complex structure and $\ip{\,\cdot\,}{\,\cdot\,}$ is the Riemannian metric.  We also assume that there exist parallel bundle maps $\pi_1:TM\to TM$ and $\pi_2: TM\to TM$ such that the following conditions are satisfied.
\begin{enumerate}
    \item Both $\pi_1$ and $\pi_2$ are orthogonal projections on each fiber.
    \item $\pi_1+\pi_2$ is the identity map on $TM$.
    \item The kernels of $\pi_1$ and $\pi_2$ on each fiber are Lagrangian subspaces.
\end{enumerate}
It follows that $\ker\pi_1$ and $\ker\pi_2$ are everywhere orthogonal, and $J$ maps one to the other.  Moreover, $J\pi_1 = \pi_2 J$ and $J\pi_2 = \pi_1 J$.  A typical example is $M = \BC^n$ (or any quotient of $\BC^n$ such as a complex torus) on which $\pi_1$ is projection from $\BC^n$ onto $\BR^n$ and $\pi_2$ is the projection from $\BC^n$ onto $J(\BR^n)$ where $J$ is the standard complex structure on $\BC^n$.

Given such a splitting structure on $TM$, one can define the following parallel $2$-tensor $S$ (see \cite{SW02}):
\begin{align}
S(X,Y) &= \ip{J \pi_1(X)}{\pi_2(Y)}
\label{defS} \end{align}
for any $X,Y\in T_p M$ at any $p\in M$.

Suppose $L_p$ is a Lagrangian subspace of $T_pM$, it is not hard to check that the restriction of $S$ to $L_p$ is symmetric, i.e. $S(X, Y)=S(Y, X)$ if $X, Y\in L_p$.  Moreover, if $\pi_1: L_p\to T_pM$ is injective, one can apply the singular value decomposition theorem to find an orthonormal basis $\{a_i\}$ for $\pi_1(T_pM)$ and real numbers $\{\ld_i\}$ such that
\begin{align}
\left\{ e_i = \frac{1}{\sqrt{1+\ld_i^2}}(a_i + \ld_iJ(a_i))\right\}_{i=1}^n
\label{frame} \end{align}
forms an orthonormal basis for $L_p$.  Note that $\{J(a_i)\}$ constitutes an orthonormal basis for $\pi_2(T_pM)$.  In terms of this basis,
\begin{align}
S_{ij} &= S(e_i,e_j) = \frac{\ld_i}{1+\ld_i^2}\dt_{ij} ~.
\label{Sld} \end{align}




\subsection{Two-Convex Lagrangians}

For a Lagrangian submanifold $F:L\hookrightarrow M$, we consider several geometric conditions that are characterized by the projection map $\pi_1$ and the tensor $S$ defined in \eqref{defS}.

\begin{defn}
    A Lagrangian submanifold $L\subset M$ is said to be \emph{graphical} if $\pi_1:T_pL\to T_pM$ is injective for any $p\in L$;
\end{defn}
A typical example is when $M=\BC^n$ and $L$ is the graph of $\nabla u$ for a function $u$ defined on $\BR^n$, $\lambda_i's$ in \eqref{frame} are exactly the eigenvalues of $D^2u$, the Hessian of $u$, see \cite{TW02}*{Section 2}.

The graphical condition can be characterized by the positivity of a geometric quantity introduced in \cite{W02}.
Fix an orientation for $\pi_1(TM)$.  Under $\pi_1^*$, the volume form of $\pi_1(TM)$ gives a parallel $n$-form on $M$, and denote it by $\Om$.  It is clear that a Lagrangian submanifold $L$ is graphical if and only if $*\Om$ is nowhere zero, where $*\Om$ denotes the Hodge star of the restriction of the $n$-form $\Om$ to $L$.  If so, orient the Lagrangian so that $*\Om > 0$.  With respect to \eqref{frame},
\begin{align}
    *\Om &= \frac{1}{\sqrt{\prod_{i=1}^n(1+\ld_i^2)}} ~.
\end{align}

We also consider the restriction of $S$ \eqref{defS} to $L$, $F^*S$, as a symmetric 2-tensor on $L$. By \eqref{Sld}, if the sum of any two eigenvalues of $F^*S$ is positive, then 
\begin{align}
S_{ii}+S_{jj} &= \frac{(\ld_i+\ld_j)(1+\ld_i\ld_j)}{(1+\ld_i^2)(1+\ld_j^2)} > 0
\label{sumS} \end{align}
for any $i\neq j$, or equivalently, $(\ld_i+\ld_j)(1+\ld_i\ld_j) > 0$ for any $i\neq j$.  The region is not a connected region, and we always focus on the connected component where
\begin{align}
\ld_i+\ld_j > 0 \,\text{ and }\, 1+\ld_i\ld_j > 0 \,\text{ for any }i\neq j ~.
\label{2pos} \end{align}

\begin{defn}
A graphical Lagrangian submanifold $L\subset M$ is said to be 
\begin{enumerate}
    \item \emph{convex} if $\lambda_i>0$ for each $i$ on $L$.
    \item \emph{two-convex} if $\ld_i+\ld_j > 0 \,\text{ and }\, 1+\ld_i\ld_j > 0 $ for any $i\neq j$, or \eqref{2pos} holds, on $L$
\end{enumerate}
\end{defn}

It is known that the Lagrangian condition is preserved by the mean curvature flow \cite{Smoczyk96}.
The main theorem of this paper is that \eqref{2pos} implies the long-time existence and convergence of the Lagrangian mean curvature flow.

\begin{thm} \label{main1}
Let $L\subset M$ be compact Lagrangian submanifold.  If $L$ is graphical and two-convex, then the mean curvature flow of $L$ exists for all time, and remains graphical and two-convex.  Moreover, it converges smoothly to a flat Lagrangian submanifold as $t\to\infty$.
\end{thm}

This theorem generalizes \cite{SW02}*{Theorem A}, which assumes $L$ is convex, or $\ld_i>0$ for all $i$. In fact, all results of Lagrangian mean curvature flows \cites{SW02,CCH12, CCY13} known to us are in the following cases: (1) the convex case, (2) cases that are equivalent to the convex case through unitary transformations (see the next subsection), or (3) cases that are perturbations of (1) and (2).

\begin{rmk}
In the proof of Theorem \ref{main1}, we implicitly assume that the ambient space $M$ is also compact.  It follows that $M$ is topologically a torus, or its finite quotient.  The theorem holds true for some non-compact $M$ as well.  For instance, if $M$ is the cotangent bundle of a flat $n$-torus, one can prove by using the distance (squared) to the zero section that the mean curvature flow of $L$ remains in a compact subset.  See for instance \cite{TW20}*{Theorem A}.
\end{rmk}

We briefly describe the steps involved the proof as follows:
\begin{enumerate}
\item We start with a compact two-convex Lagrangian that satisfies \eqref{2pos}.  We derive the evolution equation of the following quantity (see \eqref{log_det})
\[\log \prod_{i<j} \frac{(\ld_i+\ld_j)(1+\ld_i\ld_j)}{(1+\ld_i^2)(1+\ld_j^2)},\]show that it is monotone non-decreasing along the mean curvature flow, and therefore \eqref{2pos} is preserved.

\item We derive the evolution equation of 
\[\log(*\Om)=-\frac{1}{2}\log(\prod_{i=1}^n(1+\ld_i^2))\] and show that it is monotone non-decreasing along the mean curvature flow as long as $1+\lambda_i\lambda_j>0$, which was established in the last step. This in particular shows that each $\lambda_i$ remains uniformly bounded. 

\item  We prove that the second fundamental forms are bounded by contradiction. Suppose the second fundamental forms are unbounded, through a blow-up argument we obtain a non-flat ancient solution of the graphical Lagrangian mean curvature flow.
A Liouville theorem (\cite{NY11}, see Section 3)  which applies the Krylov-Safonov estimate to the equation of 
the Lagrangian angle $\theta$,
\[\theta=\sum_i^n \arctan\lambda_i,\]allows us to conclude that the ancient solution must be stationary. Finally, we apply the Bernstein theorem of \cite{TW02} that asserts any stationary solution satisfying the condition $1+\lambda_i\lambda_j>0$ must be affine and arrive at the contradiction. 

\end{enumerate}

We remark that the underlying parabolic equation is the following equation for the potential function $u$:
\begin{align}
    \frac{\pl u}{\pl t} &= \frac{1}{\sqrt{-1}} \log\frac{\det(\bfI + \sqrt{-1}D^2 u)}{\sqrt{\det(\bfI + (D^2u)^2)}} ~.
\label{LMCF_potential0} \end{align}

The estimates of $\lambda_i's$ correspond to the $C^2$ estimates of the solution $u$ and the estimates of the second fundamental forms correspond to the $C^3$ estimates. 

The convex assumption implies that the right hand side of  \eqref{LMCF_potential0} , i.e. the Lagrangian angle $\theta$, as a function of $D^2 u$ is concave in the space of symmetric matrices and thus PDE theories of fully nonlinear elliptic and parabolic equations \cites{CNS85, K87, A04} are applicable. The two-convex assumption (and the area-decreasing assumption in the next subsection) arises naturally in the study of the Lagrangian Grassmannian \cite{TW02} and the Gauss map of the mean curvature flow \cite{W03a}. It is interesting to see if some similar approach would work for related problems such as the deformed Hermitian--Yang--Mills equation considered in \cites{JY17, CJY20} or the curvature type equations considered in \cite{GZ21}. On the other hand, it is a natural question to ask if two-convexity can replace the convex assumption in the work of Caffarelli-Nirenberg-Spruck \cite{CNS85}.


\subsection{Area-Decreasing Lagrangians}

It is known that when $M=\BC^n$ the convex case $\lambda_i>0$ for each $i$ is essentially equivalent to the case $|\lambda_i|<1$ for each $i$ through a unitary transformation $U(n)$ of $\BC^n$ (\cite{TW02}, or the Lewy transformation in \cite{Y02}). The two-convex case is essentially equivalent to the following area-decreasing case through the same unitary transformation.

One can consider another parallel $2$-tensor $P$:
\begin{align}
    P(X,Y) &= \ip{\pi_1(X)}{\pi_1(Y)} - \ip{\pi_2(X)}{\pi_2(Y)} ~.
\label{defP} \end{align}
With respect to the frame \eqref{frame},
\begin{align*}
    P_{ij} &= P(e_i,e_j) = \frac{1-\ld_i^2}{1+\ld_i^2}\dt_{ij} ~.
\end{align*}
For a Lagrangian submanifold $F:L\hookrightarrow M$, $F^*P$ being $2$-positive means that
\begin{align}
    P_{ii}+P_{jj} = \frac{1-\ld_i^2\ld_j^2}{(1+\ld_i^2)(1+\ld_j^2)} > 0
\end{align}
for any $i\neq j$.

\begin{defn}
A graphical Lagrangian submanifold $L\subset M$ is said to be
\emph{area-decreasing} if 
\begin{align}\label{area-d}
    |\ld_i\ld_j| <1 \text{ for any } i \not= j
\end{align}
holds true at every $p\in L$.
\end{defn}

When $M=\BC^n$ and $L$ is the graph of $\nabla f$, the condition corresponds to $\nabla f$ as a map from $\BR^n$ to $\BR^n$ is area-decreasing. 

The same results in Theorem \ref{main1} hold true for the mean curvature flow of area-decreasing Lagrangians.

\begin{thm} \label{main2}
Let $L\subset M$ be compact Lagrangian submanifold.  If it is graphical and area-decreasing, then the mean curvature flow of $L$ exists for all time, and remains graphical and area-decreasing.  Moreover, it converges to a flat Lagrangian submanifold as $t\to\infty$.
\end{thm}

This theorem generalizes \cite{STW16}*{Theorem 2}, which assumes $\dim L = 2$.

The paper is organized as follows.  In section 2, we derive the evolution equations and provide quantitative bounds of relevant quantities.  In section 3, a Liouville Theorem for ancient solutions of Lagrangian mean curvature flows is discussed.  Section 4 is devoted to prove Theorem \ref{main1} and Theorem \ref{main2}.

\begin{ack*}The second and third authors thank Professor Smoczyk for a discussion in December 2012, in which it was observed that the argument in \cite{TW04} also implies the positivity of $S^{[2]}$ defined in \eqref{def_S2} is preserved along Lagrangian mean curvature flows.\end{ack*}


\section{Evolution Equations}

For a Lagrangian submanifold, $J$ induces an isometry between its tangent bundle and its normal bundle.  As a consequence, its second fundamental form is totally symmetric.  That is to say,
\begin{align}
h_{ijk} &= \ip{\br{\nabla}_{e_i}e_j}{J(e_k)}
\end{align} \label{second} 
is totally symmetric in $i,j,k$, where $\{e_i\}$ is an orthonormal basis for its tangent space. Here $\br{\nabla}$ is the covariant derivative of $M$.

Suppose that $F:L\times[0,T)\to M$ is a Lagrangian mean curvature flow.  From \cite{SW02}*{section 3.2}, $F^*S$ satisfies
\begin{align}
(\heat)S_{ij} &= h_{mki}h_{mk\ell}S_{\ell j} + h_{mkj}h_{mk\ell}S_{i\ell} + 2h_{ki\ell}h_{kjm}S_{\ell m} ~,
\label{evolS} \end{align}
where the equation is in terms of an evolving orthonormal frame and repeated indexes are summed. 

Since $S$ is a parallel tensor on $M$, $\br{\nabla} S=0$, the (spatial) gradient of $F^*S$ is (see for example \cite{TW04}*{p.1121})
\begin{align}
S_{ij;k} &= e_k(S(e_i,e_j)) - S(\nabla_{e_k}e_i,e_j) - S(e_i,\nabla_{e_k}e_j) \notag \\
&= S(\br{\nabla}_{e_k}e_i-\nabla_{e_k}e_i,e_j) + S(e_i, \br{\nabla}_{e_k}e_j-\nabla_{e_k}e_j),
\end{align} where $\nabla$ is the covariant derivative on $L$.
The definition of second fundamental forms \eqref{second} implies 
$\br{\nabla}_{e_k}e_i-\nabla_{e_k}e_i=h_{ki\ell} J(e_\ell)$ and 
\[S_{ij;k} = h_{ki\ell}\,S(J(e_\ell),e_j) +h_{kj\ell}\,S(e_i,J(e_\ell)).\]

At a space-time point $p$, $S_{ij;k}$ can thus be expressed in terms the frame \eqref{frame} as
\begin{align}
S_{ij;k}
= h_{kij} \left( \frac{\ld_j^2}{1+\ld_j^2} - \frac{1}{1+\ld_i^2} \right) = -\oh h_{kij} \left( \frac{1-\ld_i^2}{1+\ld_i^2} + \frac{1-\ld_j^2}{1+\ld_j^2} \right). \label{gradS}
\end{align}


\subsection{The Logarithmic Determinant of $\ST$}

In \cite{CNS85}, another tensor $\ST$ is introduced to study the two-positivity of $F^*S$; see also \cite{TW04}*{section 5}.  Similar to \cite{TTW22}, we consider the equation of the logarithmic determinant of $\ST$.

With respect to an orthonormal frame, $\ST$ is defined by
\begin{align}
\ST_{(ij)(k\ell)} &= S_{ik}\dt_{j\ell} + S_{j\ell}\dt_{ik} - S_{i\ell}\dt_{jk} - S_{jk}\dt_{i\ell} \label{def_S2}
\end{align}
for any $i<j$ and $k<\ell$.
It can be regarded as a symmetric endomorphism on $\Ld^2TL$.

At a space-time point $p$, suppose that $S$ is diagonal in terms of the frame \eqref{frame}.  It follows from \eqref{def_S2} that $\ST_{(ij)(k\ell)}|_{p} = (S_{ii} + S_{jj})\dt_{ik}\dt_{j\ell}$.  Thus, the $2$-positivity of $F^*S$ is equivalent to the positivity of $\ST$.

It can be proved that the positivity of $\ST$ is preserved along the flow in the same way as in \cite{TW04}. However, in this article we consider another quantity that also controls the two-convexity condition: the logarithmic determinant of $S^{[2]}$:
\begin{equation}\label{log_det}
\log\det(\ST)=\log \prod_{i<j} \frac{(\ld_i+\ld_j)(1+\ld_i\ld_j)}{(1+\ld_i^2)(1+\ld_j^2)}.
\end{equation}

It is a straightforward computation to show that the function $\log \det(\ST)$ satisfies
\begin{align} \begin{split}
&\quad (\heat)\log(\det \ST) \\
&\stackrel{\text{at }p}{=} \sum_{1\leq i<j\leq n} \left[ (S_{ii}+S_{jj})^{-1}(\heat)(S_{ii}+S_{jj}) +  (S_{ii}+S_{jj})^{-2}|\nabla(S_{ii}+S_{jj})|^2 \right ] \\
&\qquad + 2\sum_{1\leq i\leq n}\sum_{\substack{1\leq j<k\leq n\\j\neq i,\,k\neq i}}(S_{ii}+S_{jj})^{-1}(S_{ii}+S_{kk})^{-1}|\nabla{S_{jk}}|^2 ~.
\end{split} \label{lndetST1} \end{align}

The main task is to calculate the first term on the right hand side of \eqref{lndetST1}.  According to \eqref{evolS},
\begin{align*}
(\heat)(S_{ii}+S_{jj}) &= 2\sum_{k,\ell} \left[ h_{k\ell i}^2(S_{ii} + S_{\ell\ell}) + h_{k\ell j}^2(S_{jj} + S_{\ell\ell}) \right] \\
&= 4\sum_k\left[ h_{kii}^2\,S_{ii} + h_{kjj}^2\,S_{jj} \right] + 4(S_{ii} + S_{jj})\sum_{k}h_{kij}^2 \\
&\quad + 2\sum_{k}\sum_{\ell\neq\{i,j\}} \left[ h_{k\ell i}^2(S_{ii} + S_{\ell\ell}) + h_{k\ell j}^2(S_{jj} + S_{\ell\ell}) \right] ~.
\end{align*}
By \eqref{gradS},
\begin{align}
|\nabla(S_{ii} + S_{jj})|^2 &= \sum_k \left( h_{kii}\frac{1-\ld_i^2}{1+\ld_i^2} + h_{kjj}\frac{1-\ld_j^2}{1+\ld_j^2} \right)^2 ~.
\label{gradS1} \end{align}
It follows that
\begin{align*}
&\quad (S_{ii} + S_{jj})(\heat)(S_{ii} + S_{jj}) + |\nabla(S_{ii} + S_{jj})|^2 \\
&\qquad - 4(S_{ii} + S_{jj})^2\sum_{k}h_{kij}^2 - 2(S_{ii}+S_{jj})\sum_{k}\sum_{\ell\neq\{i,j\}} \left[ h_{k\ell i}^2(S_{ii} + S_{\ell\ell}) + h_{k\ell j}^2(S_{jj} + S_{\ell\ell}) \right] \\
&= 4(S_{ii} + S_{jj})\sum_k\left[ h_{kii}^2\,S_{ii} + h_{kjj}^2\,S_{jj} \right] + \sum_k \left( h_{kii}\frac{1-\ld_i^2}{1+\ld_i^2} + h_{kjj}\frac{1-\ld_j^2}{1+\ld_j^2} \right)^2 \\
&= \sum_k \left[ \frac{(\ld_i+\ld_j)^2 + (1+\ld_i\ld_j)^2}{(1+\ld_i^2)(1+\ld_j^2)}(h_{kii}^2 + h_{kjj}^2) + \frac{(1-\ld_i^2)(1-\ld_j^2)}{(1+\ld_i^2)(1+\ld_j^2)}2h_{kii}h_{kjj} \right] \\
&= \sum_{k} \left[ \frac{(1+\ld_i\ld_j)^2}{(1+\ld_i^2)(1+\ld_j^2)}(h_{kii}+h_{kjj})^2 + \frac{(\ld_i+\ld_j)^2}{(1+\ld_i^2)(1+\ld_j^2)}(h_{kii}-h_{kjj})^2 \right] ~.
\end{align*}
Hence,
\begin{align*}
&\quad  (S_{ii}+S_{jj})^{-1}(\heat)(S_{ii}+S_{jj}) +  (S_{ii}+S_{jj})^{-2}|\nabla(S_{ii}+S_{jj})|^2 \\
&= 4\sum_{k}h_{kij}^2 + 2(S_{ii}+S_{jj})^{-1}\sum_{k}\sum_{\ell\neq\{i,j\}} \left[ h_{k\ell i}^2(S_{ii} + S_{\ell\ell}) + h_{k\ell j}^2(S_{jj} + S_{\ell\ell}) \right] \\
&\quad +  \sum_{k} \left[ \frac{(1+\ld_i^2)(1+\ld_j^2)}{(\ld_i+\ld_j)^2}(h_{kii}+h_{kjj})^2 + \frac{(1+\ld_i^2)(1+\ld_j^2)}{(1+\ld_i\ld_j)^2}(h_{kii}-h_{kjj})^2 \right] ~.
\end{align*}

With \eqref{lndetST1}, it leads to the following Proposition.
\begin{prop} \label{prop_lndetST}
Suppose that a graphical Lagrangian mean curvature flow is two-convex, then the function $\log(\det\ST)$ satisfies
\begin{align}
&\quad (\heat)\log(\det\ST) \notag \\
&\geq \sum_{i<j}\sum_{k} \left[ 4h_{kij}^2 + \frac{(1+\ld_i^2)(1+\ld_j^2)}{(\ld_i+\ld_j)^2}(h_{kii}+h_{kjj})^2 + \frac{(1+\ld_i^2)(1+\ld_j^2)}{(1+\ld_i\ld_j)^2}(h_{kii}-h_{kjj})^2 \right] \label{lndetST2} \\
&\geq 2|A|^2 \geq 0 ~. 
\label{lndetST2a}
\end{align}

In particular, $\min \log(\det\ST)$ is monotone non-decreasing along the flow and two-convexity is preserved. 
\end{prop}

\begin{proof}
It remains to show that \eqref{lndetST2a} is no less than $2|A|^2$ under the condition \eqref{2pos}.  It is straightforward to verify that under the condition \eqref{2pos},
\begin{align}
\frac{(1+\ld_i^2)(1+\ld_j^2)}{(\ld_i+\ld_j)^2} \geq 1 \quad\text{and}\quad \frac{(1+\ld_i^2)(1+\ld_j^2)}{(1+\ld_i\ld_j)^2} \geq 1 ~.
\end{align}
Note that both equalities are attained when $\ld_i = 1 = \ld_j$.  Hence, the right hand side of \eqref{lndetST2} is no less than
\begin{align}
\sum_{i<j}\sum_{k} \left[ 4h_{kij}^2 + 2h_{kii}^2 + 2h_{kjj}^2\right] = 2|A|^2 + 2(n-2)\sum_{i,k}h_{kii}^2 ~.
\end{align}
It finishes the proof of this proposition.
\end{proof}



\subsection{The Logarithmic Determinant of the Jacobian of $\pi_1$}

In the minimal Lagrangian case, the equation for $\log(*\Om)$ is derived in \cite{TW02}*{(2.4)}.  The computation in the parabolic case is essentially the same, and
\begin{align*}
    (\heat)\log(*\Om) &= \sum_{i,j,k}h_{ijk}^2 + \sum_{i,j}\ld_i^2\,h_{iij}^2 + 2\sum_{\substack{i,j,k\\i<j}}\ld_i\ld_j\,h_{ijk}^2 ~.
\end{align*}
By re-grouping the summations, one finds the following proposition.
\begin{prop} \label{prop_lnOm}
Along a graphical Lagrangian mean curvature flow, the function $\log(*\Om)$ satisfies
\begin{align} \begin{split}
(\heat)\log(*\Om) &= \sum_{i}(1+\ld_i^2)h_{iii}^2 + \sum_{i\neq j}(3+\ld_i^2 + 2\ld_i\ld_j)h_{iij}^2 \\
    &\quad + \sum_{i<j<k}(6+2\ld_i\ld_j+2\ld_j\ld_k+2\ld_k\ld_i)h_{ijk}^2 ~.
\end{split} \label{logdetOm} \end{align}
If the flow is in addition two-convex, then $(\heat)\log(*\Om) \geq 0$ and $\min \log(*\Om) $ in monotone non-decreasing along the flow. 
\end{prop}
\begin{proof}
It is clear that the right hand side is non-negative if $1+\lambda_i\lambda_j>0$ which is part of the two-convexity assumption. 

\end{proof}


\subsection{Some Quantitative Bounds}

Since $\frac{\ld}{1+\ld^2}$ takes value within $[-\oh,\oh]$, the expression \eqref{sumS}, $S_{ii} + S_{jj}$, is always no greater than $1$.  It follows that for a two-convex Lagrangian, $\det\ST$ takes value within $(0,1]$.  Hence, $\log(\det\ST)\in(-\infty,0]$.  Since $*\Om = 1/\sqrt{\prod_{i}(1+\ld_i^2)}$, $\log(*\Om)\in(-\infty,0]$.

For a two-convex Lagrangian submanifold, suppose that $\log(*\Om) \geq -\dt_1$ and $\log(\det\ST) \geq -\dt_2$ for some $\dt_1,\dt_2>0$.  It follows that
\begin{align}
    \sum_i\ld_i^2 &\leq e^{2\dt_1} - 1
\label{ld_bound} \end{align}
for all $i$.  From $\log(\det\ST) \geq -\dt_2$,
\begin{align*}
    \frac{(\ld_i+\ld_j)(1+\ld_i\ld_j)}{(1+\ld_i^2)(1+\ld_j^2)} &\geq \prod_{k<\ell}(S_{kk}+S_{\ell\ell}) \geq e^{-\dt_2}
\end{align*}
for any $i\neq j$.  Therefore,
\begin{align}
    {(\ld_i+\ld_j)(1+\ld_i\ld_j)} &\geq e^{-\dt_2} ~.
\label{convex_bound} \end{align}
Under the condition \eqref{2pos}, \eqref{ld_bound} and \eqref{convex_bound} lead to that
\begin{align}
    1+\ld_i\ld_j \geq \frac{e^{-\dt_2}}{\sqrt{2(e^{2\dt_1}-1)}}
    \quad\text{and}\quad \ld_i+\ld_j \geq \frac{2e^{-\dt_2}}{e^{2\dt_1}+1}
\label{convex_bound1} \end{align}
for any $i\neq j$.


\section{A Liouville Theorem}

In this section, we state a Liouville theorem for ancient solutions of the Lagrangian mean curvature flow in $\BC^n\equiv\BR^{2n}$ under the graphical condition. For discussions of ancient solutions of the Lagrangian mean curvature flow under other assumptions, see \cite{LLS21}.   The theorem
is due to Nguyen and Yuan \cite{NY11}*{Proposition 2.1} and is a direct consequence of the Krylov--Safonov estimate \cite{K87}. We include the proof here for completeness. 

\begin{thm} \label{thm_Liouville}
Let $u$ be a smooth solution to
\begin{align}
    \frac{\pl u}{\pl t} &= \frac{1}{\sqrt{-1}} \log\frac{\det(\bfI + \sqrt{-1}D^2 u)}{\sqrt{\det(\bfI + (D^2u)^2)}}
\label{LMCF_potential} \end{align}
in $Q = \BR^n\times(-\infty,t_0]$ for some $t_0>0$.  Denote by $\ld_1,\ldots,\ld_n$ the eigenvalues of the Hessian of $u$, $D^2u$.  Suppose that every $|\ld_i|$ is bounded over $Q$.  
Then, $u$ is stationary, i.e.\ $u$ satisfies the special Lagrangian equation. 
\end{thm}

\begin{proof}
Denote the right hand side of \eqref{LMCF_potential} by $\ta$.  It is the argument of the complex number $\det(\bfI+\sqrt{-1}D^2u)$.  Note that $\ta$ is a smooth function over $Q$, and takes value within $(-n\pi/2,n\pi/2)$.  According to \cite{HL82}*{\S III.2.D}, the differential of $\ta$ is equivalent to the mean curvature of the gaph of $Du$.  As a consequence, \eqref{LMCF_potential} means that the Lagrangian $\{(x,Du)\}$ evolves under the mean curvature flow.  

The induced metric on the graph of $Du$ has the first fundamental form given by
\begin{align*}
    g &= \bfI + (D^2u)^2 = (\bfI+\sqrt{-1}D^2u)(\bfI-\sqrt{-1}D^2u) ~.
\end{align*}
In particular, $(\bfI+\sqrt{-1}D^2u)^{-1}$ is $(\bfI-\sqrt{-1}D^2u)g^{-1} = g^{-1}(\bfI-\sqrt{-1}D^2u)$.  With this understood, the derivative of \eqref{LMCF_potential} in $t$ gives
\begin{align}
    \frac{\pl\ta}{\pl t} &= g^{ij}\,\pl_t u_{ij} = g^{ij}\,\pl_i\pl_j(\pl_t u) = g^{ij}\,\pl_i\pl_j\ta ~,
\label{Lag_angle} \end{align}
where $[g^{ij}]$ is the inverse of $g = \bfI + (D^2u)^2$.  Since $|\ld_i|$'s are uniformly bounded, \eqref{Lag_angle} is a uniformly parabolic equation.  As the right hand side of \eqref{Lag_angle} has no first and zeroth order terms,  the Krylov--Safonov estimate \cite{K87}*{Lemma 2 on p.133} implies that there exist positive $\af$ and $C$  depending on $n$, $\sup_{Q}|u|$ and $\sup_{Q}|D^2u|$ such that
\begin{align*}
    \sup\left\{ \frac{|\ta(\td{x},\td{t})-\ta(x,t)|}{\max\{|\td{x}-x|^\af,|\td{t}-t|^{\frac{\af}{2}}\}} : (\td{x},\td{t}),(x,t)\in B_r\times[t_0-r^2,t_0] ~,~ (\td{x},\td{t})\neq(x,t) \right\} &\leq C\,\frac{1}{r^\af}
\end{align*}
for any $r>0$.  By letting $r\to\infty$, one finds that $\ta(x,t) = \ta(\td{x},\td{t})$ for any $(x,t)\neq(\td{x},\td{t})$.

It follows that the graph of $Du$ is a time-independent minimal/special Lagrangian submanifold.   It finishes the proof of this theorem.
\end{proof}


\section{Proof of the Main Theorems}

This section is devoted to the proof of Theorem \ref{main1}.  The proof of Theorem \ref{main2} is almost the same, and we only address the key ingredient at the end of this section.


\subsection{Preserving the Graphical and Two-Convexity Condition} \label{sec_preserve}

Suppose that $L$ is a compact, oriented $n$-dimensional manifold, and $F_0:L\to M$ is a two-convex Lagrangian submanifold.  Consider the Lagrangian mean curvature flow $F:L\times[0,T)\to M$ with $F(\cdot,0) = F_0(\cdot)$, where $T$ is the maximal existence time.  Let
\begin{align*}
    \bar{\tau} &= \sup\{\tau\in(0,T) : \text{the flow remains two-convex in }[0,\tau) \} ~.
\end{align*}

Denote by $L_t$ the image of $L\times\{t\}$ under $F$.
Since $L$ is compact, $\log(\det\ST) \geq -\dt_2$ and $\log(*\Om) \geq -\dt_1$ for some $\dt_1,\dt_2>0$ on $L_0$.  Due to Proposition \ref{prop_lndetST}, Proposition \ref{prop_lnOm} and the maximum principle, both $\min_{L_t}\log(\det\ST)$ and $\min_{L_t}\log(*\Om)$ are non-decreasing in $t\in[0,\bar{\tau})$.  If $\bar{\tau} < T$, it follows from \eqref{ld_bound} and \eqref{convex_bound} that $L\times\{\bar{\tau}\}$ is two-convex.  Because of the openness of the two-convexity condition, this is a contradiction, and $\bar{\tau}$ must be the maximal existence time.


\subsection{Long-time Existence} \label{sec_long_time}

With \eqref{lndetST2a}, one may use the same argument as that in the proof of \cite{W02}*{Theorem A} to prove the mean curvature flow exists for all time.  It is based on Huisken's monotonicity formula \cite{Huisken90} and White's regularity theorem \cite{White05}.

Below, we present another argument based on the Liouville theorem (Theorem \ref{thm_Liouville}).  Recall that Huisken proved that if the maximal existence time $T < \infty$, then it is characterized by
\begin{align}
    \limsup_{t\to T}\max_{L_t}|A|^2 = \infty ~;
\label{sing_criterion} \end{align}
see \cite{Huisken84}*{Theorem 8.1} for the hypersurface case.

Assume \eqref{sing_criterion} for some $T < \infty$.  There exist sequences $\{t_k\}\in (0,T)$ and $\{\bx_k\}\in L$ such that
\begin{itemize}
    \item $t_k\to T$ monotonically as $k\to\infty$;
    \item $|A|^2(\bx_k,t_k) = \max\{|A|^2(x,t) : (\bx,t)\in L\times[0,t_k]\}$;
    \item $|A|^2(\bx_k,t_k)\to \infty$ monotonically as $k\to\infty$.
\end{itemize}
Denote $|A|(\bx_k,t_k)$ by $\rho_k$.  Due to \cite{STW16}*{Lemma 3.1} (see also \cite{LL92}*{Theorem 1}), the second fundamental form obeys
\begin{align}
    (\heat)|A|^2 &\leq -2|\nabla A|^2 + 3|A|^4 ~.
\label{second_fund} \end{align}
In particular, the right hand side is no greater than $3|A|^4$. By the maximum principle with the initial time $t_k$ and the criterion \eqref{sing_criterion} of Huisken, one finds that the flow exists at least for $t < t_k + \frac{1}{3\rho_k^2}$, and
\begin{align*}
    \max_{L_t}|A|^2 &\leq \left( \frac{1}{\rho_k^2} - 3(t-t_k) \right)^{-1} ~.
\end{align*}
Since $T$ is the maximal existence time, $ T-t_k \geq \frac{1}{3\rho_k^{2}}$ for all $k\in\BN$.  Identify a neighborhood of $F(\bx_k,t_k)$ with a subset of $\BR^{2n}$, and consider
\begin{align}
    \td{F}_k(\cdot,s) &= \rho_k\,\left[ F(\cdot,t_k+\frac{s}{\rho_k^2}) - F(\bx_k,t_k) \right] ~.
\label{rescale} \end{align}
The image of $\td{F}_k$ is given by the graph of $D\td{u}_k$ for some $\td{u}_k:\CU_k\subset\BR^n\times\BR\to\BR$ with $\td{u}_k(0,0) = 0$ and $D\td{u}_k(0,0)=0$.  It follows from $\rho_k\to\infty$ that any compact subset of $\BR^n\times(-\infty,c)$ is contained in $\CU_k$ for any $k>\!>1$.  By the standard blow-up argument\footnote{Since $|\ld_i|'s$ are uniformly bounded, so is $D^2\td{u}_k$.  The third order derivative $D^3\td{u}_k$ is equivalent to the second fundamental form, which is bounded.  The higher order derivatives are also bounded; see \cite{Baker10}*{Proposition 4.8}.}, $\td{u}_k$ converges to $u:\BR^n\times(-\infty,c)\to\BR$ satisfying \eqref{LMCF_potential}, and the convergence is smooth on any compact subset of $\BR^n\times(-\infty,c)$.

Note that the slope is invariant under the rescaling \eqref{rescale}, and $\ld_i$'s remain unchanged.  In particular, the eigenvalues of $D^2u$ satisfy \eqref{ld_bound} and \eqref{convex_bound1} everywhere.  Hence, Theorem \ref{thm_Liouville} implies that the graph of $Du$ is a 
special Lagrangian submanifold that satisfies the condition that for any $i,j$, $3 + 2\ld_i\ld_j \geq \dt$ over $Q$;
we conclude that the graph of $Du$ in $\BR^{2n}$, $\{(x,D u):x\in\BR^n\}$, is an affine $n$-plane by \cite{TW02}*{Corollary C}.  However, the second fundamental form of the graph of $D\td{u}_k$ has norm $1$ at $(0,0)$.  This is a contradiction.


\subsection{Convergence}

The key to conclude the convergence as $t\to\infty$ is to show that $\max_{L_t}|A|^2\to0$ as $t\to\infty$.


\subsubsection{Uniform Boundedness of $|A|^2$}
The first task is to show that $|A|^2$ is uniformly bounded.  Suppose not, then
\begin{align}
    \limsup_{t\to\infty}\max_{L_t}|A|^2 &= \infty ~.
\end{align}
With the same argument as that in section \ref{sec_long_time}, one can extract a blow-up limit, which is a non-trivial ancient solution to \eqref{LMCF_potential}.  By the same token, it contradicts to Theorem \ref{thm_Liouville}.


\subsubsection{$L^2$-Convergence of $|A|^2$} \label{sec_L2}
With the uniformly boundedness of $|A|^2$, \cite{Baker10}*{Proposition 4.8} asserts that $|\nabla^\ell A|^2$ is uniformly bounded for all $\ell\in\BN$.

Denote $\log(\det\ST)$ by $v$, whose value belongs to $[-\dt_2,0)$.  By Proposition \ref{prop_lndetST}, it obeys $(\heat)v \geq 2|A|^2$.  Denote by $\dd\mu_t$ the volume form of $L_t$.  Since $\frac{\pl}{\pl t}\dd\mu_t = -|H|^2\dd\mu_t$,
\begin{align*}
    2\int_{L_t}|A|^2\,\dd\mu_t &\leq \int_{L_t} \left[ (\heat)v - v|H|^2 \right]\dd\mu_t = \frac{\dd}{\dd t}\int_{L_t} v\,\dd\mu_t ~.
\end{align*}
This together with $|\int_{L_t}v\,\mu_t| \leq \dt_2 \vol(L_t) \leq \dt_2 \vol(L_0)$ implies that
\begin{align}
    \int_0^\infty (\int_{L_t}|A|^2\dd\mu_t)\,\dd t &< \infty ~.
\label{L2_1} \end{align}

By \eqref{second_fund} and the uniform boundedness of $|A|^2$ and $|\nabla A|^2$,
\begin{align}
    \frac{\dd}{\dd t}\int_{L_t}|A|^2\dd\mu_t &= \int_{L_t} \left[ (\heat)|A|^2 - |A|^2|H|^2 \right]\dd\mu_t \notag \\
    &\leq \int_{L_t} \left[ 3|A|^4 - 2|\nabla A|^2 - |A|^2|H|^2 \right]\dd\mu_t \leq c_1 ~. \label{L2_2}
\end{align}
According to \cite{TW20}*{Lemma 6.3}, \eqref{L2_1} and \eqref{L2_2} imply that $\int_{L_t}|A|^2\dd\mu_t\to0$ as $t\to\infty$.


\subsubsection{Convergence of the Flow} \label{sec_ptwise}
Fix $t\geq 0$; suppose that $\max_{L_t}|A|^2$ is achieved at $\bx_0$.  On a fixed size neighborhood of $\bx_0$, $L_t$ is the graph over $\pi_1(T_{\bx_0}M)$, whose higher derivatives are all bounded.  It follows that there exists a $c_2>0$ such that $\int_{L_t}|A|^2\dd\mu_t \geq c_2\max_{L_t}|A|^2$.  Therefore,
\begin{align}
    \lim_{t\to0}\max_{L_t}|A|^2 = 0 ~.
\end{align}

Since the mean curvature flow is a gradient flow and the metrics are analytic, it follows from the theorem of Simon \cite{Simon83} that the flow converges to a unique limit as $t\to\infty$.  This finishes the proof of Theorem \ref{main1}.


\subsection{About Theorem \ref{main2}}

Analogous to \eqref{def_S2}, we introduce $P^{[2]}$ to study the $2$-positivity of $F^*P$.  According to \cite{TTW22}*{Theorem 3.2}, the logarithmic determinant\footnote{$P$ is denoted as $S$ in \cite{TTW22}.} of $P^{[2]}$ obeys
\begin{align}
    (\heat)\log(\det P^{[2]}) &\geq 2|A|^2
\label{log_det_PT} \end{align}
along the mean curvature flow. In terms of $\lambda_i$, we have 
\begin{align}
    \log\det(P^{[2]})=\log \prod_{i<j} \frac{1-\lambda_i^2\lambda_j^2}{(1+\ld_i^2)(1+\ld_j^2)} ~.
\label{log_det_P} \end{align}

The proof of Theorem \ref{main2} is very similar to that of Theorem \ref{main1}, and is sketched below.  As in section \ref{sec_preserve}, denote by $T$ the maximal existence time.  Let
\begin{align*}
    \bar{\tau} &= \sup\{\tau\in(0,T) : \text{the flow remains graphical and area-decreasing in }[0,\tau) \} ~.
\end{align*}
By the maximum principle on \eqref{log_det_PT}, $\log(\det P^{[2]})$ is uniformly bounded from below.  It follows from \cite{TTW22}*{Lemma 3.3} that $L_t$ remains graphical and area-decreasing as long as the flow exists.

For the long-time existence, suppose that $T<\infty$, and perform the same blow-up argument as that in \ref{sec_long_time} to get a non-trivial ancient solution of \eqref{LMCF_potential}.  Here, we rely on \cite{W03}*{Theorem 1.1} to conclude that any entire, graphical minimal submanifold that satisfies the condition  $|\ld_i\ld_j|\leq 1-\dt$ must be an affine $n$-plane.  It is a contradiction, and thus $T=\infty$.

By a similar blow-up argument, the second fundamental form cannot tend to infinity as $t\to\infty$.  As in section \ref{sec_L2}, one deduces that $\int_{L_t}|A|^2\dd\mu_t\to 0$ as $t\to\infty$ by considering the integration of \eqref{log_det_PT} over $L_t$.  The same argument as that in section \ref{sec_ptwise} implies that $\sup_{L_t}|A|^2\to 0$ as $t\to\infty$.  Finally, one invokes the theorem of \cite{Simon83} to finish the proof.

\end{document}